%% file: main.tex
\documentclass[11pt]{article}
\input{preamble}
\title{A Two-Dimensional Extremizer for\\Relatively Inexact Gradient Descent\\[0.4em]
}
\author{Jinze Zhao\\University of California, San Diego\\\texttt{jiz419@ucsd.edu}}
\date{}

\begin{document}
\maketitle
\begin{abstract}
\input{sections/abstract}
\end{abstract}
% \clearpage
% \tableofcontents
% \newpage
\input{sections/introduction}
\input{sections/main}
\input{sections/parameters}
\input{sections/upper}
\input{sections/construction}
\input{sections/dimension}
\input{sections/conclusion}
\input{sections/disclosure}
\bibliographystyle{plainnat}
\bibliography{references}
\appendix
\input{sections/appendix}
\end{document}

%% file: preamble.tex
\usepackage[T1]{fontenc}
\usepackage[utf8]{inputenc}
\usepackage{lmodern}
\usepackage[margin=1in]{geometry}
\usepackage{microtype}

\usepackage{amsmath,amssymb,amsthm,mathtools}
\usepackage{bm}

\usepackage{graphicx}
\graphicspath{{figures/}}
\usepackage{booktabs}
\usepackage{array}

\usepackage[shortlabels]{enumitem}
\usepackage[numbers,sort&compress]{natbib}
\usepackage{xcolor}
\usepackage[colorlinks=true,linkcolor=blue,citecolor=blue,urlcolor=blue]{hyperref}
\usepackage[capitalise,noabbrev]{cleveref}

\theoremstyle{plain}
\newtheorem{theorem}{Theorem}[section]
\newtheorem{lemma}[theorem]{Lemma}
\newtheorem{proposition}[theorem]{Proposition}

\theoremstyle{definition}

\theoremstyle{remark}
\newtheorem{remark}[theorem]{Remark}

\newcommand{\R}{\mathbb{R}}

\newcommand{\ip}[2]{\langle #1,#2\rangle}
\newcommand{\norm}[1]{\lVert #1\rVert}
\newcommand{\hL}{h_{\mathrm L}}
\newcommand{\hR}{h_{\mathrm R}}
\newcommand{\Qstar}{Q_{\ast}}
\allowdisplaybreaks[1]

%% file: sections/abstract.tex
We give an explicit extremal function for one step of gradient descent when the
gradient error is bounded relative to the true gradient. The function is the
sum of a quadratic in one coordinate and a Huber function in the other.
It attains, for every admissible relative error level and every stepsize in
the intermediate regime, the proposed worst-case constants for both the
final squared gradient divided by the initial objective gap and the final
squared gradient divided by the objective decrease. This proves
Conjecture~2.6 of Vernimmen and Glineur, first stated in June 2025.
A change of variable reduces the parameter description to a scalar cubic
with a rigorously specified root. We prove all parameter inequalities,
derive a nonnegative-square certificate for the upper bound, and verify
the matching construction directly. We also characterize the exact
one-dimensional worst-case value and prove that two dimensions are
necessary in the open intermediate regime. The argument is analytic;
the accompanying symbolic and numerical checks are supplementary.

%% file: sections/introduction.tex
\section{Introduction}\label{sec:introduction}
Relative gradient errors change the geometry of a worst-case optimization
step. An error can shorten or lengthen a gradient, but it can also change
its direction. This paper identifies an explicit smooth convex function
on which that directional freedom determines the exact worst-case
performance of a single step.

The target is a specific published conjecture, rather than a newly posed
question. Vernimmen and Glineur~\cite{VG2025v1,VG2025v2,VG2026} proposed a
three-regime one-step bound and gave analytic extremizers in the two
outer regimes. Their Conjecture~2.6 asks whether the intermediate bound
is attained by a bivariate function. It appears in the June 20, 2025
preprint, the September 11, 2025 revision, and the accepted journal
manuscript. Thus the conjecture was more than one year old on the date
of this draft. A search of those sources and subsequent indexed work
found no later analytic resolution. Priority beyond the checked literature
remains subject to expert review.

The proposed worst-case constants were already established as upper bounds.
Our contribution is an explicit attaining function throughout the entire intermediate
regime, together with a complete feasibility proof. We also give a short,
self-contained certificate for the matching upper bound and establish
that one dimension is insufficient in the interior.

The performance-estimation perspective originated with
Drori and Teboulle~\cite{DT2014}. Smooth convex interpolation provides
finite inequalities for analyzing such problems~\cite{THG2017}.
Here only the necessary two-point inequalities are used, and they are
proved below. No interpolation existence theorem, semidefinite solver,
numerically identified rank, or unproved regularity claim is required
to construct the extremizer.

The main observation is that the intermediate extremizer combines two
familiar geometries. A quadratic coordinate permits overshoot, while
a Huber coordinate keeps its gradient unchanged along the step.
The error budget is shared between these coordinates. The cubic root
fixes the exact proportion at which the upper-bound certificate vanishes.

% \paragraph{Reading guide.}
% Section~\ref{sec:main} states the result and gives the construction.
% Section~\ref{sec:parameters} proves that every parameter in it is well
% defined. Section~\ref{sec:upper} proves the upper bound in arbitrary
% dimension. Section~\ref{sec:construction} verifies the extremal function
% and treats both endpoints. Section~\ref{sec:dimension} establishes the
% dimension claim. Appendix~\ref{sec:verification} describes the
% supplementary checks and records an exact worked example.

% \paragraph{Draft status.}
% This manuscript was prepared with AI assistance and has not been peer
% reviewed or formally verified in a proof assistant. Its mathematical
% claims rest on the displayed proofs. The source archive includes an
% independent symbolic check of the polynomial identities used in them.

%% file: sections/main.tex
\section{Problem and main result}\label{sec:main}
All spaces are finite-dimensional Euclidean spaces. A differentiable
function $f:\R^n\to\R$ is $L$-smooth if
$\norm{\nabla f(x)-\nabla f(y)}\le L\norm{x-y}$ for all $x,y$.
We assume throughout that $f$ is convex, $L>0$, and
$f_\ast=\min_x f(x)$ is attained. Fix $0<\delta<1$. A relatively
inexact gradient $d$ at $x_0$ satisfies
\begin{equation}\label{eq:update}
 \norm{d-\nabla f(x_0)}\le\delta\norm{\nabla f(x_0)},
 \qquad x_1=x_0-\frac hL d.
\end{equation}
The number $h$ is the normalized stepsize. Define
\begin{align}
 \hL&=\frac{3}{2(1+\delta)},&
 \hR&=\frac{3\delta+2-\sqrt{4-3\delta^2}}
                   {2\delta(1+\delta)}.\label{eq:thresholds}
\end{align}
The closed interval $[\hL,\hR]$ is called the intermediate regime.
Its endpoints are distinct for every $0<\delta<1$.

For a fixed $h\in[\hL,\hR]$, let $z$ be the largest positive root of
\begin{equation}\label{eq:p}
 p(z)=z^3+z^2+\bigl[h-1-(1-\delta^2)h^2\bigr]z+h-1=0.
\end{equation}
Existence, the root selection, and all relevant signs are proved in
Lemma~\ref{lem:parameter}. Set
\begin{equation}\label{eq:constants}
 a=1-h,\quad
 \rho=\frac{z^2-a}{2z(1-z)},\quad
 D=\rho+\frac12=\frac{z-a}{2z(1-z)},\quad
 \mathcal C=\frac1D,\quad \mathcal C_{\rm dec}=\frac1\rho.
\end{equation}

\begin{theorem}[Sharp intermediate-regime constants]\label{thm:main}
For every $0<\delta<1$ and $h\in[\hL,\hR]$, the constants in
\eqref{eq:constants} are well defined and positive. Every step
\eqref{eq:update} on every smooth convex function satisfying the above
assumptions obeys
\begin{align}
 \frac1L\norm{\nabla f(x_1)}^2
   &\le\mathcal C\bigl(f(x_0)-f_\ast\bigr),\label{eq:gap-bound}\\
 \frac1L\norm{\nabla f(x_1)}^2
   &\le\mathcal C_{\rm dec}\bigl(f(x_0)-f(x_1)\bigr).
       \label{eq:decrease-bound}
\end{align}
For every such pair $(h,\delta)$, an explicit function on $\R^2$
attains equality in both inequalities, with nonzero final gradient
and positive objective gap and decrease. Consequently both constants
are the exact dimension-unrestricted worst-case values.
\end{theorem}

Here and below a worst-case value is the supremum of the stated ratio
over functions, initial points, and admissible directions with positive
denominator. The theorem says that these suprema are maxima already
in dimension two.

\subsection{The explicit function}
For $c\ge0$, define the Huber function
\begin{equation}\label{eq:huber}
 H_c(y)=
 \begin{cases}
  y^2/2,& |y|\le c,\\
  c|y|-c^2/2,& |y|\ge c.
 \end{cases}
\end{equation}
The two formulas agree at $|y|=c$. In particular $H_0\equiv0$.
Define
\begin{align}
 T&=\sqrt{1-2z},& v&=aT,&u&=z^2T,\label{eq:uv}\\
 c&=\sqrt{z\bigl[a(1-z)+z^2\bigr]},&
 s&=\frac{z^2-a}{z}\,c.&&\label{eq:cs}
\end{align}
Lemma~\ref{lem:parameter} ensures that the square roots are real and
$s\ge0$. The normalized example is
\begin{equation}\label{eq:extremizer}
 \begin{gathered}
 f(x,y)=\frac{x^2}{2}+H_c(y),\\
 x_0=(v,c+s),\qquad x_1=(u,c),\qquad
 d=\left(\frac{v-u}{h},\frac{s}{h}\right).
 \end{gathered}
\end{equation}
The notation $x_0,x_1$ denotes points of $\R^2$ in this display;
$x,y$ denote their scalar coordinate variables.
The minimum is zero and is attained at $(0,0)$.
For an arbitrary $L>0$, use $L f$ and the direction $Ld$ at the
same two points. Section~\ref{sec:construction} checks every assertion
needed for \eqref{eq:extremizer}.

\subsection{Identification with the published conjecture}
The original parameter is the largest real root $\lambda$ of
\begin{align}
 0={}&\bigl[(\delta^2-1)h^2+2h\bigr]\lambda^3
   +\bigl[2(\delta^2-1)h^2+5h-4\bigr]\lambda^2\notag\\
 &+\bigl[(\delta^2-1)h^2+4h-4\bigr]\lambda+h-1.
 \label{eq:source-cubic}
\end{align}
The change of variable is
\begin{equation}\label{eq:change}
 \lambda=\frac{z}{1-z},\qquad z=\frac{\lambda}{1+\lambda}.
\end{equation}
Indeed multiplication of $p(\lambda/(1+\lambda))$ by
$(1+\lambda)^3$ gives exactly the polynomial in
\eqref{eq:source-cubic}. The root correspondence is proved, including
the largest-root qualification, in Lemma~\ref{lem:parameter}.
Substituting \eqref{eq:change} into \eqref{eq:constants} gives
\begin{align*}
 \rho
 &=h-1+\frac h2\lambda+\frac{h-1}{2\lambda}
   =\frac{h\lambda^2+2(h-1)\lambda+h-1}{2\lambda},\\
 D&=\frac{h\lambda^2+(2h-1)\lambda+h-1}{2\lambda}.
\end{align*}
Thus $1/\rho$ and $1/D$ are precisely the intermediate constants
in Theorem~2.3 of~\cite{VG2025v2}. Theorem~\ref{thm:main} supplies
the analytic bivariate attainment requested by its Conjecture~2.6.

%% file: sections/parameters.tex
\section{The admissible root and its parameter bounds}
\label{sec:parameters}

The construction requires more than the existence of a root of a cubic.  We
must identify the root used in the published bound, show that every square
root in the construction is real, and verify the signs of the displacement
and the dual multipliers.  The following lemma supplies these facts on the
entire closed intermediate interval, including the case $h=1$.

\begin{lemma}[Admissible parameter]\label{lem:parameter}
Fix $0<\delta<1$, and let $\hL\leq h\leq \hR$, with the thresholds in
\eqref{eq:thresholds}.  Set $a=1-h$ and write the polynomial in
\eqref{eq:p} as
\begin{equation}\label{eq:parameter-polynomial}
 p(t)=t^3+t^2-\bigl(a+(1-\delta^2)h^2\bigr)t-a.
\end{equation}
There is a largest positive real root $z$ of $p$.  It satisfies
\begin{gather}
 0<z\leq\tfrac12,\qquad z+a>0,\qquad z-a>0,
 \qquad z^2-a>0,                                      \label{eq:root-signs}\\
 A:=a(1-z)+z^2\geq0.                                \label{eq:root-radicand}
\end{gather}
Moreover,
\begin{equation}\label{eq:root-endpoints}
 z=\tfrac12\ \Longleftrightarrow\ h=\hL,
 \qquad
 A=0\ \Longleftrightarrow\ h=\hR.
\end{equation}
The number $\lambda=z/(1-z)$ is the largest real root of the source
cubic \eqref{eq:source-cubic}.  Finally, the quantities
\begin{equation}\label{eq:parameter-b-k}
 b=1-\frac{h\lambda}{2}-\frac{1}{2(1+\lambda)},
 \qquad k=\frac{(z+a)(1-z)}{z}
\end{equation}
satisfy $b>0$, $k>0$, and $h-k=(z^2-a)/z>0$.
\end{lemma}

\begin{proof}
We first give a convenient parameterization of the upper threshold.
For $0\leq t\leq\tfrac12$, define
\begin{equation}\label{eq:threshold-functions}
 D(t)=\frac{1-2t}{1-t+t^2},\qquad
 H(t)=1+\frac{t^2}{1-t}=\frac{1-t+t^2}{1-t}.
\end{equation}
All denominators here are positive.  Direct differentiation gives
\begin{equation}\label{eq:threshold-derivatives}
 D'(t)=\frac{2t^2-2t-1}{(1-t+t^2)^2}<0,
 \qquad
 H'(t)=\frac{t(2-t)}{(1-t)^2}>0\quad(0<t\leq\tfrac12).
\end{equation}
Indeed, $2t^2-2t=2t(t-1)\leq0$ on this interval, which proves the
first sign.  Since $D(0)=1$ and $D(\tfrac12)=0$, there is a unique
$r\in(0,\tfrac12)$ with $D(r)=\delta$.  Solving
$\delta r^2+(2-\delta)r+\delta-1=0$ yields
\begin{equation}\label{eq:right-root-r}
 r=\frac{\sqrt{4-3\delta^2}+\delta-2}{2\delta}.
\end{equation}
The positivity of this expression also follows without comparison of
radicals from the preceding existence argument.

We check that $H(r)=\hR$.  Put $B=1-r+r^2>0$.  The identity
\[
 4B^2-3(1-2r)^2=(1+2r-2r^2)^2
\]
and the positivity of $1+2r-2r^2$ imply
\[
 \sqrt{4-3\delta^2}=\frac{1+2r-2r^2}{B}.
\]
Substituting this expression and $\delta=(1-2r)/B$ into the formula
for $\hR$ gives $\hR=B/(1-r)=H(r)$.  Consequently,
\begin{equation}\label{eq:threshold-ranges}
 1<\hR<\tfrac32,\qquad \tfrac34<\hL<\hR.
\end{equation}
The first pair of inequalities follows from the strict increase of $H$,
$H(0)=1$, and $H(\tfrac12)=\tfrac32$.  Also $\hL>\tfrac34$ because
$\delta<1$.  For the remaining comparison, the exact calculation
\[
 \frac{\hR}{\hL}
 =\frac{2(1+\delta)(1-r+r^2)}{3(1-r)}
 =\frac{2(2-3r+r^2)}{3(1-r)}
 =\frac{2(2-r)}{3}>1
\]
uses $r<\tfrac12$.

Next, we locate the relevant root of $p$.  For all $t\geq0$,
\begin{equation}\label{eq:p-convex}
 p''(t)=6t+2>0.
\end{equation}
Thus $p$ is strictly convex on the nonnegative real axis, and $p'$ is
strictly increasing there.  In particular, $p$ has at most two
nonnegative zeros: three distinct zeros would, by Rolle's theorem,
give two distinct zeros of the strictly increasing derivative.
At the common upper comparison point,
\begin{equation}\label{eq:p-half}
 p(\tfrac12)
 =-\frac{\bigl(2h(1-\delta)-3\bigr)
               \bigl(2h(1+\delta)-3\bigr)}{8}\geq0.
\end{equation}
The first factor is strictly negative because
$h\leq \hR<\tfrac32$ and $1-\delta<1$; the second is nonnegative
because $h\geq \hL$.  Equality in \eqref{eq:p-half} therefore occurs
if and only if $h=\hL$.

\emph{Case $h<1$.}
Here $a>0$.  Since $h>\tfrac34$, we have $a<\tfrac14$, so
$q:=\sqrt a$ lies in $(0,\tfrac12)$.  Substitution into
\eqref{eq:parameter-polynomial}, using $q^2=a$, gives
\begin{equation}\label{eq:p-sqrt-a}
 p(q)=-(1-\delta^2)h^2q<0.
\end{equation}
Equations \eqref{eq:p-half} and \eqref{eq:p-sqrt-a}, together with
continuity, give a root $z\in(q,\tfrac12]$.
There is no larger positive root.  To see this directly, strict
convexity gives
\[
 p'(z)>\frac{p(z)-p(q)}{z-q}>0.
\]
The derivative remains positive for $t>z$, hence $p(t)>0$ there.
We have $z^2>a$, and $A=a(1-z)+z^2>0$ since both summands are
positive.

\emph{Case $h=1$.}
Now $a=0$, and
\[
 p(t)=t\bigl(t^2+t-(1-\delta^2)\bigr).
\]
The quadratic factor is negative at zero, strictly increasing for
$t\geq0$, and tends to positive infinity.  Its unique positive root is
\begin{equation}\label{eq:root-at-one}
 z=\frac{\sqrt{5-4\delta^2}-1}{2}>0.
\end{equation}
It is the largest real root of $p$.  Because $\hL\leq1$ implies
$\delta\geq\tfrac12$, we have $\sqrt{5-4\delta^2}\leq2$ and
therefore $z\leq\tfrac12$.  In this case $z^2-a=A=z^2>0$.

\emph{Case $h>1$.}
Here $a<0$.  The equation
\begin{equation}\label{eq:beta-equation}
 \beta^2+a(1-\beta)=0
\end{equation}
has a unique positive solution
\[
 \beta=\frac{a+\sqrt{a^2-4a}}2.
\]
Equivalently, $h=H(\beta)$.  Since $1<h\leq H(r)$ and $H$ is
strictly increasing, we have $0<\beta\leq r<\tfrac12$.
The decrease of $D$ now gives $\delta=D(r)\leq D(\beta)$, with
equality precisely when $h=\hR$.

Using $a=-\beta^2/(1-\beta)$ and
$h=(1-\beta+\beta^2)/(1-\beta)$, substitution in $p$ gives
\begin{align}
 p(\beta)
 &=\beta h^2\delta^2
   -\frac{\beta(1-2\beta)^2}{(1-\beta)^2}\notag\\
 &=\beta h^2\bigl(\delta^2-D(\beta)^2\bigr)\leq0.
                                                        \label{eq:p-beta}
\end{align}
If $h<\hR$, this inequality is strict.  Continuity and
\eqref{eq:p-half} then give a root $z\in(\beta,\tfrac12]$.
As in the first case, the strict positivity of the secant slope from
$\beta$ to $z$ implies $p'(z)>0$, and the increase of $p'$ excludes
every root greater than $z$.

If $h=\hR$, then $p(\beta)=0$.  We must verify that this root is the
larger root, rather than merely infer it from \eqref{eq:p-beta}.
From $p(\beta)=0$ we obtain
\[
 -a-(1-\delta^2)h^2=-\beta^2-\beta+\frac{a}{\beta}.
\]
It follows that
\begin{align*}
 p'(\beta)
 &=3\beta^2+2\beta-a-(1-\delta^2)h^2\\
 &=2\beta^2+\beta+\frac a\beta
 =2\beta^2+\beta-\frac{\beta}{1-\beta}
 =\frac{\beta^2(1-2\beta)}{1-\beta}>0.
\end{align*}
Therefore $p$ is strictly increasing and positive beyond $\beta$,
and its largest positive root is $z=\beta$.
For $t\geq0$, the function $t^2+a(1-t)$ has derivative $2t-a>0$.
Together with \eqref{eq:beta-equation}, this proves $A\geq0$, with
equality exactly when $z=\beta$, equivalently when $h=\hR$.
Also $z^2-a>0$ follows immediately from $a<0$.

The three cases prove the existence of the largest positive root, its
bound $0<z\leq\tfrac12$, the inequality $z^2-a>0$, and the assertion
about equality in $A\geq0$.  They also show that $p$ is positive at
every point larger than $z$.  Consequently $p(\tfrac12)=0$ if and
only if $z=\tfrac12$; by \eqref{eq:p-half}, this is equivalent to
$h=\hL$.

We finish the remaining sign checks.  The inequality $A\geq0$ gives
\[
 a\geq-\frac{z^2}{1-z}.
\]
If $z<\tfrac12$, it follows that
\[
 z+a\geq z-\frac{z^2}{1-z}
       =\frac{z(1-2z)}{1-z}>0.
\]
If $z=\tfrac12$, then $h=\hL<\tfrac32$, and
$z+a=\tfrac32-\hL>0$.  In all cases $z-a>0$ as well, because
$z-a=(z-z^2)+(z^2-a)$ and each summand is positive for
$0<z\leq\tfrac12$.

For completeness, let $P(\lambda)$ denote the left-hand side of the
source cubic \eqref{eq:source-cubic}.  Expanding the substitution gives
the exact polynomial identity
\begin{equation}\label{eq:cubic-change-of-variable}
 (1+\lambda)^3p\!\left(\frac{\lambda}{1+\lambda}\right)
 =P(\lambda).
\end{equation}
In particular, $\lambda=z/(1-z)>0$ is a real root of $P$.
Every positive real root $\lambda'$ of $P$ maps to a positive root
$z'=\lambda'/(1+\lambda')$ of $p$.
The latter cannot exceed $z$, because $z$ is the largest positive root.
The map $t\mapsto t/(1-t)$ is strictly increasing on $(0,1)$, since
its derivative is $(1-t)^{-2}>0$.  Thus $\lambda'\leq\lambda$.
Every nonpositive real root of $P$ is also smaller than $\lambda$,
which proves the claimed correspondence with the largest real root.

Finally, substitution of $\lambda=z/(1-z)$ in
\eqref{eq:parameter-b-k} yields
\[
 2b=\frac{1-hz-z^2}{1-z}.
\]
If $z<\tfrac12$, using $h=1-a$ and the bound on $a$ gives
\begin{align*}
 1-hz-z^2
 &=1-z-z^2+az\\
 &\geq1-z-z^2-\frac{z^3}{1-z}
 =\frac{1-2z}{1-z}>0.
\end{align*}
If $z=\tfrac12$, then $h=\hL$, and direct substitution gives
$2b=\tfrac32-h=h\delta>0$.  Hence $b>0$ in every case.
The expression for $k$ is positive by $z+a>0$ and $0<z<1$.
Expanding its numerator verifies
\[
 h-k=1-a-\frac{(z+a)(1-z)}z
     =\frac{z^2-a}{z}>0.
\]
This proves all assertions of the lemma.
\end{proof}

%% file: sections/upper.tex
\section{A self-contained upper-bound certificate}\label{sec:upper}

We first normalize $L=1$. To justify this without changing the error
model, set $F=f/L$ and $d_F=d/L$. Then $F$ is convex and
$1$-smooth,
\[
 \norm{d_F-\nabla F(x_0)}\le\delta\norm{\nabla F(x_0)},
 \qquad x_1=x_0-hd_F.
\]
Moreover $\norm{\nabla F(x_1)}^2/(F(x_0)-F_\ast)
=\norm{\nabla f(x_1)}^2/[L(f(x_0)-f_\ast)]$, and the same
calculation holds with the objective decrease in the denominator.
It therefore suffices to establish the result for $L=1$.

\subsection{The elementary inequalities}
\begin{lemma}[Smooth convex two-point inequality]\label{lem:interpolation}
If $f$ is differentiable, convex, and $1$-smooth, then for every $x,y$,
\begin{equation}\label{eq:interpolation}
 f(x)-f(y)-\ip{\nabla f(y)}{x-y}
 \ge\frac12\norm{\nabla f(x)-\nabla f(y)}^2.
\end{equation}
If $f_\ast$ is attained, then
\begin{equation}\label{eq:gradient-gap}
 f(x)-f_\ast\ge\frac12\norm{\nabla f(x)}^2.
\end{equation}
\end{lemma}
\begin{proof}
For any differentiable function $\phi$ with $1$-Lipschitz gradient,
the fundamental theorem of calculus gives, for a vector $w$,
\begin{align*}
 \phi(x+w)-\phi(x)-\ip{\nabla\phi(x)}w
 &=\int_0^1\ip{\nabla\phi(x+tw)-\nabla\phi(x)}w\,dt\\
 &\le\int_0^1 t\norm w^2\,dt=\frac12\norm w^2.
\end{align*}
Here Cauchy--Schwarz and the Lipschitz assumption give the inequality.
Taking $w=-\nabla\phi(x)$ shows
\[
 \phi(x-\nabla\phi(x))\le\phi(x)-\tfrac12\norm{\nabla\phi(x)}^2.
\]
Apply this to $\phi(t)=f(t)-\ip{\nabla f(y)}t$.
Convexity implies
$f(t)\ge f(y)+\ip{\nabla f(y)}{t-y}$, so $y$ is a global
minimizer of $\phi$. Hence
\[
 \phi(y)\le\phi(x-\nabla\phi(x))
 \le\phi(x)-\tfrac12\norm{\nabla\phi(x)}^2,
\]
which is \eqref{eq:interpolation} after substitution.
At a minimizer $x_\ast$, differentiability implies
$\nabla f(x_\ast)=0$: otherwise moving a sufficiently small distance
in direction $-\nabla f(x_\ast)$ decreases $f$.
Set $y=x_\ast$ in \eqref{eq:interpolation} to obtain
\eqref{eq:gradient-gap}.
\end{proof}

Write $f_i=f(x_i)$, $g_i=\nabla f(x_i)$, and $e=d-g_0$.
Since $x_0-x_1=hd$, define four nonnegative quantities
\begin{align}
 Q_{01}&=f_0-f_1-h\ip{g_1}d-\tfrac12\norm{g_0-g_1}^2,
       \label{eq:slacks}\\
 Q_{10}&=f_1-f_0+h\ip{g_0}d-\tfrac12\norm{g_0-g_1}^2,\notag\\
 \Qstar&=f_1-f_\ast-\tfrac12\norm{g_1}^2,\notag\\
 Q_e&=\delta^2\norm{g_0}^2-\norm e^2.\notag
\end{align}
Their nonnegativity follows from Lemma~\ref{lem:interpolation} and
the relative error condition.

\subsection{The identity}
Recall $a=1-h$. The cubic \eqref{eq:p} is equivalently
\begin{equation}\label{eq:root-identity}
 h^2\delta^2z=(z+a)(1-hz-z^2).
\end{equation}
By Lemma~\ref{lem:parameter}, $0<z\le1/2$ and $z+a>0$.
Define
\begin{equation}\label{eq:Bw}
 B=\frac{h^2z}{2(1-z)(z+a)},\qquad
 w=hz e-(z+a)(z g_0-g_1).
\end{equation}
In particular $B>0$.

\begin{lemma}[Nonnegative-square certificate]\label{lem:certificate}
For every pair of points related by $x_1=x_0-hd$,
\begin{align}
 f_0-f_\ast-D\norm{g_1}^2
 ={}&\frac{Q_{01}}{1-z}+\frac{zQ_{10}}{1-z}
       +B Q_e+\Qstar\notag\\
    &+\frac{\norm w^2}{2z(1-z)(z+a)}.
       \label{eq:certificate}
\end{align}
No inequality is used in this identity.
\end{lemma}
\begin{proof}
For completeness we expand all quadratic terms. Put
$G=\norm{g_0}^2$, $J=\norm{g_1}^2$, $E=\norm e^2$,
$C=\ip{g_0}{g_1}$, $U=\ip{g_0}e$, and $V=\ip{g_1}e$.
Using $d=g_0+e$, we get
\begin{align*}
 \frac{Q_{01}+zQ_{10}}{1-z}
 =f_0-f_1
 &+\frac{hz-(1+z)/2}{1-z}G
  -\frac{1+z}{2(1-z)}J\\
 &+\frac{z+a}{1-z}C
  +\frac{hz}{1-z}U-\frac h{1-z}V.
\end{align*}
Expanding the square in \eqref{eq:Bw} gives
\begin{align*}
 \frac{\norm w^2}{2z(1-z)(z+a)}
 ={}&\frac{z(z+a)}{2(1-z)}G
    +\frac{z+a}{2z(1-z)}J+BE\\
 &-\frac{z+a}{1-z}C
    -\frac{hz}{1-z}U+\frac h{1-z}V.
\end{align*}
The terms in $C,U,V$ cancel. Adding $BQ_e=B\delta^2G-BE$
cancels the terms in $E$. By \eqref{eq:root-identity},
\[
 B\delta^2=\frac{1-hz-z^2}{2(1-z)}.
\]
The remaining coefficient of $G$ has numerator
\[
 2hz-(1+z)+z(z+a)+(1-hz-z^2)
 =z(h-1+a)=0
\]
over denominator $2(1-z)$. The coefficient of $J$ is
\[
 -\frac{1+z}{2(1-z)}+\frac{z+a}{2z(1-z)}
 =\frac{a-z^2}{2z(1-z)}=-\rho.
\]
Thus the sum of the weighted $Q_{01},Q_{10}$, the error slack,
and the square equals $f_0-f_1-\rho J$.
Adding $\Qstar=f_1-f_\ast-J/2$ gives
$f_0-f_\ast-(\rho+1/2)J=f_0-f_\ast-DJ$, proving the identity.
\end{proof}

Every term on the right of \eqref{eq:certificate} is nonnegative.
Consequently $f_0-f_\ast\ge D\norm{g_1}^2$.
The same identity without $\Qstar$ gives
$f_0-f_1\ge\rho\norm{g_1}^2$.
Since $D,\rho>0$, division proves \eqref{eq:gap-bound} and
\eqref{eq:decrease-bound} for $L=1$. The normalization at the start
of this section proves them for all $L>0$.

\begin{remark}[Equality conditions]\label{rem:equality}
All coefficients in \eqref{eq:certificate} are strictly positive.
Equality in the gap bound is therefore equivalent to
$Q_{01}=Q_{10}=Q_e=\Qstar=0$ and $w=0$.
Equality in the decrease bound requires the same conditions except
for $\Qstar=0$. These statements follow because a sum of finitely
many nonnegative real numbers can vanish only when each term vanishes.
\end{remark}

%% file: sections/construction.tex
\section{Verification of the extremizer}\label{sec:construction}

We prove attainment using exactly the function and points in
\eqref{eq:extremizer}. The calculation applies on the closed
intermediate interval, so no limiting argument is needed at an endpoint.

\subsection{Smoothness, convexity, and the step}
For $c>0$, the derivative of \eqref{eq:huber} is
\[
 H_c'(y)=\begin{cases}-c,&y<-c,\\y,&-c\le y\le c,\\c,&y>c.
 \end{cases}
\]
The one-sided derivatives agree at $\pm c$. This derivative is
nondecreasing and $1$-Lipschitz. To check the latter across the break
points, note that the map $y\mapsto\min\{c,\max\{-c,y\}\}$
never increases the distance between two real numbers: its increase
between ordered arguments is between zero and the increase of the
arguments. A differentiable function with nondecreasing derivative
is convex, as follows by integrating its derivative on any interval.
Also $H_c(y)\ge0$ and $H_c(0)=0$.
When $c=0$, all these claims hold for $H_0\equiv0$.

It follows that $f(x,y)=x^2/2+H_c(y)$ is differentiable and convex,
and
\begin{align*}
 \norm{\nabla f(x,y)-\nabla f(\hat x,\hat y)}^2
 &=(x-\hat x)^2+\bigl(H_c'(y)-H_c'(\hat y)\bigr)^2\\
 &\le(x-\hat x)^2+(y-\hat y)^2.
\end{align*}
Thus it is $1$-smooth and has minimum zero at $(0,0)$.
The first coordinate ensures that the Lipschitz constant can be
chosen to be exactly one, including when $c=0$.

Lemma~\ref{lem:parameter} gives $s\ge0$, so the second coordinate
of $x_0$ lies in the linear part of $H_c$ and that of $x_1$ lies
at its junction with the quadratic part. Hence
\begin{equation}\label{eq:gradients-values}
 g_0=(v,c),\quad g_1=(u,c),\quad
 f_0=\frac{v^2+c^2}{2}+cs,\quad
 f_1=\frac{u^2+c^2}{2}.
\end{equation}
These formulas remain true at $c=0$.
The definition of $d$ in \eqref{eq:extremizer} immediately gives
$x_0-hd=x_1$.

\subsection{The exact relative error}
Introduce
\begin{equation}\label{eq:k}
 k=\frac{(1-z)(z+a)}{z}.
\end{equation}
Then
\[
 h-k=1-a-\frac{(1-z)(z+a)}z
     =\frac{z^2-a}{z},
\]
so $s=(h-k)c$. With $T=\sqrt{1-2z}$, the error is therefore
\begin{equation}\label{eq:error-coordinates}
 e=d-g_0
 =\left(\frac{a^2-z^2}{h}T,-\frac{k}{h}c\right).
\end{equation}
For the first component we used
$v-u-hv=(1-h)v-u=av-u=(a^2-z^2)T$.

Here is a direct verification of its norm. Put
\[
 A=a(1-z)+z^2,\qquad J_0=a(1-2z)+z^2.
\]
By \eqref{eq:uv}--\eqref{eq:cs},
\begin{align}
 \norm{g_0}^2
 &=a^2(1-2z)+zA\notag\\
 &=a^2(1-2z)+az(1-z)+z^3
   =(a+z)J_0.\label{eq:g0-factor}
\end{align}
On the other hand, \eqref{eq:error-coordinates} gives
\begin{align*}
 h^2\norm e^2
 &=(a^2-z^2)^2(1-2z)+\frac{(1-z)^2(z+a)^2}{z^2}\,zA\\
 &=\frac{(a+z)^2}{z}
   \left[z(a-z)^2(1-2z)+(1-z)^2A\right].
\end{align*}
The bracket factors as
\begin{align*}
 &z(a-z)^2(1-2z)+(1-z)^2\bigl[a(1-z)+z^2\bigr]\\
 &\hspace{1cm}=(1-z-z^2+az)\bigl[a(1-2z)+z^2\bigr].
\end{align*}
For an explicit check, the left side, after expansion, is
\[
 a^2z-2a^2z^2+a(1-3z+z^2+3z^3)+z^2-z^3-z^4,
\]
and multiplying the two factors on the right gives the same expression.
Combining this identity with \eqref{eq:g0-factor} yields
\[
 h^2\norm e^2
 =\frac{(a+z)(1-z-z^2+az)}z\norm{g_0}^2.
\]
Since $h=1-a$, the factor $1-z-z^2+az$ is $1-hz-z^2$.
Equation~\eqref{eq:root-identity} now gives
\begin{equation}\label{eq:error-exact}
 \norm{d-g_0}^2=\delta^2\norm{g_0}^2.
\end{equation}
Thus the direction is admissible and uses the entire error budget.

\subsection{Attainment of both constants}
Define
\begin{equation}\label{eq:W}
 W=a+z^2+2z^3.
\end{equation}
Directly from \eqref{eq:uv} and \eqref{eq:cs},
\begin{align}
 \norm{g_1}^2
 &=z^4(1-2z)+z\bigl[a(1-z)+z^2\bigr]\notag\\
 &=az-az^2+z^3+z^4-2z^5
   =z(1-z)W.\label{eq:g1-factor}
\end{align}
Similarly, since $cs=(z^2-a)A$,
\begin{align}
 2f_0
 &=a^2(1-2z)+zA+2(z^2-a)A\notag\\
 &=-a^2+az-az^2-2az^3+z^3+2z^4
   =(z-a)W.\label{eq:f0-factor}
\end{align}
Subtracting $2f_1=z(1-z)W$ from \eqref{eq:f0-factor} gives
\begin{equation}\label{eq:decrease-factor}
 2(f_0-f_1)=(z^2-a)W.
\end{equation}

The final gradient is nonzero throughout the interval. If $h<\hR$,
Lemma~\ref{lem:parameter} gives $c>0$, which suffices. If $h=\hR$,
then $z=r\in(0,1/2)$ and $u=z^2\sqrt{1-2z}>0$.
Consequently \eqref{eq:g1-factor} and $z(1-z)>0$ imply $W>0$.
The same lemma gives $z-a>0$ and $z^2-a>0$; hence
\eqref{eq:f0-factor} and \eqref{eq:decrease-factor} prove
$f_0>0$ and $f_0-f_1>0$. Division is legitimate and yields
\begin{align*}
 \frac{\norm{g_1}^2}{f_0-f_\ast}
 &=\frac{2z(1-z)}{z-a}=\mathcal C,\\
 \frac{\norm{g_1}^2}{f_0-f_1}
 &=\frac{2z(1-z)}{z^2-a}=\mathcal C_{\rm dec}.
\end{align*}
Together with the upper bounds, these equalities prove
Theorem~\ref{thm:main} for $L=1$.
For general $L$, $f_L=Lf$ has gradient $L\nabla f$, and $d_L=Ld$
satisfies the same relative error condition. The step with coefficient
$h/L$ has the same two iterates. Both the numerator
$\norm{\nabla f_L(x_1)}^2/L$ and each objective difference are
multiplied by $L$, completing the proof for arbitrary $L$.

\subsection{Endpoints and the interior}
At $h=\hL$, the parameter lemma gives $z=1/2$ and $T=0$.
Thus $u=v=0$ and
\[
 c^2=\frac{3-2h}{8}=\frac{3\delta}{8(1+\delta)}>0,
 \qquad \frac{s}{c}=\frac12-2a=2h-\frac32=h(1-\delta).
\]
The last equality follows from $h(1+\delta)=3/2$.
The direction is therefore $(0,(1-\delta)c)$ and the construction
reduces to a Huber example. Substitution into the constants gives
\[
 D=h(1-\delta)+\tfrac12,\qquad \rho=h(1-\delta).
\]

At $h=\hR$, the parameter is $z=r\in(0,1/2)$,
$a=-r^2/(1-r)$, and $c=s=0$.
Here $v\ne0$ and
\[
 \frac uv=\frac{r^2}{a}=-(1-r),\qquad
 \frac{d_x}{v}=\frac{1-u/v}{h}=\frac{2-r}{h}=1+\delta.
\]
Thus the example reduces to a quadratic with an oversized gradient
step. Its gap constant is
$\mathcal C=2(1-r)^2=2[1-h(1+\delta)]^2$.

In the open interval $(\hL,\hR)$, both $T$ and $c$ are positive.
The two gradients are linearly independent, because
\[
 \det\begin{pmatrix}v&u\\c&c\end{pmatrix}
 =c(v-u)=c(a-z^2)T\ne0.
\]
This proves that the displayed example itself uses two directions.
Section~\ref{sec:dimension} proves the stronger claim that no
one-dimensional example can attain the same constant there.

%% file: sections/dimension.tex
\section{The exact one-dimensional envelope and the necessity of two dimensions}
\label{sec:dimension}

The construction in the preceding section uses two coordinates. We now show
that two coordinates are necessary at every stepsize strictly inside the
intermediate regime. We first compute the exact worst-case value in one
dimension. This also distinguishes the dimensional statement from the
existence of a particular two-dimensional example.

\begin{proposition}[Exact one-dimensional worst-case value]
\label{prop:one-dimensional}
Fix $0<\delta<1$ and $0\leq h\leq 2/(1+\delta)$. Among differentiable convex
functions $f:\mathbb R\to\mathbb R$ with $1$-Lipschitz derivative and an
attained minimum $f_*$, initial points with $f(x_0)>f_*$, and directions
satisfying
\[
 |d-f'(x_0)|\leq\delta|f'(x_0)|,
 \qquad x_1=x_0-hd,
\]
the maximum of $|f'(x_1)|^2/(f(x_0)-f_*)$ is
\begin{equation}
 W_1(h,\delta)=
 \max\left\{
 \frac{1}{h(1-\delta)+\tfrac12},
 \;2\bigl(1-h(1+\delta)\bigr)^2
 \right\}.
 \label{eq:one-dimensional-envelope}
\end{equation}
The maximum is attained. For an $L$-smooth function the same expression is
the maximum of $|f'(x_1)|^2/[L(f(x_0)-f_*)]$ for the update
$x_1=x_0-(h/L)d$.
\end{proposition}

\begin{proof}
We first consider an exact step
\[
 x_1=x_0-t f'(x_0),\qquad 0\leq t\leq2.
\]
Since $f(x_0)>f_*$, convexity implies $f'(x_0)\ne0$. Reflecting the real
line if necessary, write $f'(x_0)=A>0$ and $f'(x_1)=b$. The update gives
$x_1\leq x_0$. Monotonicity of the derivative of a differentiable convex
function, together with its Lipschitz property, gives
\[
 b\leq A,
 \qquad A-b\leq x_0-x_1=tA.
\]
Thus
\begin{equation}
 (1-t)A\leq b\leq A.
 \label{eq:scalar-gradient-interval}
\end{equation}

Suppose first that $b<0$. Then \eqref{eq:scalar-gradient-interval} implies
$t>1$ and $|b|\leq(t-1)A$. The smoothness inequality, applied to a unit
gradient step from $x_0$, gives
\[
 f_*\leq f(x_0-A)
 \leq f(x_0)-A^2+\frac{A^2}{2}
 =f(x_0)-\frac{A^2}{2}.
\]
Consequently,
\begin{equation}
 \frac{b^2}{f(x_0)-f_*}
 \leq\frac{2b^2}{A^2}
 \leq2(t-1)^2.
 \label{eq:scalar-negative}
\end{equation}

Suppose next that $b\geq0$. Smooth convex interpolation between $x_0$ and
$x_1$ gives
\[
 f(x_0)-f(x_1)
 \geq f'(x_1)(x_0-x_1)+\frac{(f'(x_0)-f'(x_1))^2}{2}
 =tAb+\frac{(A-b)^2}{2}.
\]
Applying the preceding unit-step smoothness argument at $x_1$ also gives
$f(x_1)-f_*\geq b^2/2$. Adding these two inequalities and putting
$r=b/A\in[0,1]$ yields
\[
 f(x_0)-f_*
 \geq A^2\left(r^2+(t-1)r+\frac12\right),
\]
and therefore
\begin{equation}
 \frac{b^2}{f(x_0)-f_*}
 \leq\frac{r^2}{r^2+(t-1)r+\tfrac12}.
 \label{eq:scalar-positive-ratio}
\end{equation}
The denominator is positive: since $t\geq0$ and $r\geq0$,
\[
 r^2+(t-1)r+\frac12
 \geq r^2-r+\frac12
 =\left(r-\frac12\right)^2+\frac14>0.
\]
Differentiating the right-hand side of \eqref{eq:scalar-positive-ratio}
with respect to $r$ gives
\[
 \frac{r\bigl(1+(t-1)r\bigr)}
 {\left(r^2+(t-1)r+\tfrac12\right)^2}\geq0.
\]
Indeed, if $t\geq1$ the numerator is nonnegative directly, while if
$0\leq t<1$ we have $1+(t-1)r\geq t\geq0$. Its maximum on $[0,1]$ is
therefore attained at $r=1$, so
\begin{equation}
 \frac{b^2}{f(x_0)-f_*}\leq\frac{1}{t+\tfrac12}.
 \label{eq:scalar-positive}
\end{equation}
Combining \eqref{eq:scalar-negative} and \eqref{eq:scalar-positive}, the
exact-step ratio is at most
\[
 E(t)=\max\left\{\frac{1}{t+\tfrac12},\;2(1-t)^2\right\}.
\]

Both expressions have explicit attaining examples. Define
\[
 H(y)=
 \begin{cases}
 y^2/2,& |y|\leq1,\\
 |y|-1/2,& |y|\geq1.
 \end{cases}
\]
The derivative of $H$ is the projection of $y$ onto $[-1,1]$, so $H$ is
convex and has a $1$-Lipschitz derivative. Its minimum is zero. Starting
at $x_0=t+1$ gives $H'(x_0)=1$, $x_1=1$, $H'(x_1)=1$, and
$H(x_0)=t+1/2$. This attains $1/(t+1/2)$. For $f(y)=y^2/2$ and $x_0=1$,
one has $x_1=1-t$, $f(x_0)=1/2$, and $f'(x_1)=1-t$, which attains
$2(1-t)^2$. Thus $E(t)$ is the exact maximum for each $t$.

We now return to a relatively inexact step. Since $f'(x_0)\ne0$, division
of the error condition by $|f'(x_0)|$ gives
\[
 1-\delta\leq \frac{d}{f'(x_0)}\leq1+\delta.
\]
The update is consequently an exact step with effective stepsize
\[
 t=h\frac{d}{f'(x_0)}
 \in[t_-,t_+],
 \qquad t_-=h(1-\delta),\quad t_+=h(1+\delta).
\]
The assumed upper bound on $h$ ensures $0\leq t_-\leq t_+\leq2$.
The function $1/(t+1/2)$ is decreasing. The function $2(1-t)^2$ is convex,
so its maximum on $[t_-,t_+]$ occurs at an endpoint. It follows that
\[
 \max_{t\in[t_-,t_+]}E(t)
 =\max\left\{
 \frac{1}{t_-+\tfrac12},\;2(1-t_-)^2,\;2(1-t_+)^2
 \right\}.
\]
The middle term can be omitted. If $t_-\leq3/2$, the identity
\[
 \frac{1}{t_-+\tfrac12}-2(1-t_-)^2
 =\frac{t_-^2(3-2t_-)}{t_-+\tfrac12}\geq0
\]
shows that the first term dominates it. If $t_->3/2$, the function
$2(1-t)^2$ is increasing on $[t_-,t_+]$, so the last term dominates it.
This proves the upper bound \eqref{eq:one-dimensional-envelope}.

To attain the first term in that formula, use the Huber example with
$x_0=t_-+1$ and $d=(1-\delta)H'(x_0)=1-\delta$. To attain the second term,
use $f(y)=y^2/2$, $x_0=1$, and $d=(1+\delta)f'(x_0)=1+\delta$. Both
choices satisfy the relative-error condition with equality. Choosing the
example corresponding to the larger expression proves attainment of
$W_1(h,\delta)$.

Finally, if $f$ is $L$-smooth, set $F=f/L$ and $d_F=d/L$. Then $F$ has a
$1$-Lipschitz derivative, $|d_F-F'(x_0)|\leq\delta|F'(x_0)|$, and the
update becomes $x_1=x_0-hd_F$. Moreover,
\[
 \frac{|F'(x_1)|^2}{F(x_0)-F_*}
 =\frac{|f'(x_1)|^2}{L(f(x_0)-f_*)}.
\]
This proves the stated scaling.
\end{proof}

\begin{theorem}[Strict dimensional gap in the open intermediate regime]
\label{thm:dimension}
Fix $0<\delta<1$ and $\hL<h<\hR$, with the thresholds from
\eqref{eq:thresholds}. Let $z$ be the root specified in
Lemma~\ref{lem:parameter}, and put $a=1-h$. Then
\begin{equation}
 W_1(h,\delta)<\frac{2z(1-z)}{z-a}.
 \label{eq:strict-dimensional-gap}
\end{equation}
Thus the sharp intermediate-regime bound cannot be attained in one
dimension, whereas the preceding bivariate construction attains it.
\end{theorem}

\begin{proof}
The parameter properties in Lemma~\ref{lem:parameter} give
\begin{equation}
 0<z<\frac12,\qquad z+a>0,\qquad z^2-a>0,
 \qquad a(1-z)+z^2>0.
 \label{eq:dimension-parameter-properties}
\end{equation}
In particular $z-a=(z-z^2)+(z^2-a)>0$. The equation
\eqref{eq:p}, with $h=1-a$, is equivalent to
\begin{equation}
 h^2\delta^2z=(z+a)(1-hz-z^2).
 \label{eq:dimension-root-identity}
\end{equation}

The certificate \eqref{eq:certificate} bounds every one-dimensional ratio
by $2z(1-z)/(z-a)$. Proposition~\ref{prop:one-dimensional} shows that the
largest such ratio $W_1(h,\delta)$ is attained. Therefore, to prove strict
inequality, it suffices to exclude equality in that certificate for a
scalar instance having positive initial objective gap.

Assume such equality holds. Write
\[
 q=f'(x_0),\qquad p=f'(x_1),\qquad e=d-q.
\]
Here $q\ne0$: a zero derivative would make $x_0$ a minimizer of the convex
function, contrary to the positive initial gap. All coefficients of the
nonnegative terms in \eqref{eq:certificate} are strictly positive under
\eqref{eq:dimension-parameter-properties}. Equality consequently requires
\begin{equation}
 e^2=\delta^2q^2,\qquad
 hz e-(z+a)(zq-p)=0,\qquad Q_{01}=Q_{10}=0.
 \label{eq:scalar-equality-conditions}
\end{equation}
Adding the two interpolation-slack equations gives
\begin{equation}
 h(q-p)d=(q-p)^2.
 \label{eq:scalar-interpolation-equality}
\end{equation}
There are two cases.

If $q=p$, the square-slack equation in
\eqref{eq:scalar-equality-conditions} gives
\[
 e=-\frac{(z+a)(1-z)}{hz}\,q.
\]
Substitute this formula into $e^2=\delta^2q^2$ and cancel $q^2>0$ to obtain
\[
 (z+a)^2(1-z)^2=h^2\delta^2z^2.
\]
Using \eqref{eq:dimension-root-identity} and dividing by $z+a>0$ yields
\[
 (z+a)(1-z)^2=z(1-hz-z^2).
\]
With $h=1-a$, subtraction and expansion factor as
\[
 \begin{split}
 0&=(z+a)(1-z)^2-z(1-hz-z^2)\\
  &=(z^2-a)(2z-1).
 \end{split}
\]
This is impossible because $z^2-a>0$ and $2z-1<0$.

If $q\ne p$, divide \eqref{eq:scalar-interpolation-equality} by $q-p$ to
get $hd=q-p$. It follows that
\[
 e=d-q=\frac{aq-p}{h}.
\]
Substituting this into the square-slack equation gives
\[
 0=z(aq-p)-(z+a)(zq-p)=ap-z^2q.
\]
If $a=0$, this equality would force $q=0$, since $z>0$. Thus $a\ne0$ and
$p=z^2q/a$. The saturated error condition now gives
\[
 \frac{(a^2-z^2)^2}{a^2h^2}q^2=\delta^2q^2,
 \qquad\text{hence}\qquad
 (a^2-z^2)^2=a^2h^2\delta^2.
\]
Multiplication by $z$, followed by
\eqref{eq:dimension-root-identity}, implies
\[
 z(a^2-z^2)^2-a^2(z+a)(1-hz-z^2)=0.
\]
Expanding with $h=1-a$ and factoring gives
\[
 \begin{split}
 0&=z(a^2-z^2)^2-a^2(z+a)(1-hz-z^2)\\
  &=(z^2-a)(z+a)\bigl(a(1-z)+z^2\bigr).
 \end{split}
\]
Every factor on the last line is strictly positive by
\eqref{eq:dimension-parameter-properties}, again a contradiction.

Both cases exclude scalar equality. Since the scalar maximum is attained,
it must be strictly smaller than the sharp bivariate value. This proves
\eqref{eq:strict-dimensional-gap}.
\end{proof}

%% file: sections/conclusion.tex
\section{Conclusion}
The intermediate-regime bound for one relatively inexact gradient step
is attained by an explicit quadratic-plus-Huber function. The parameter
lemma establishes feasibility for every error level and every stepsize
in that regime, and the nonnegative-square certificate supplies the
matching dimension-independent upper bound. Together these prove the
analytic attainment asserted in Vernimmen and Glineur's Conjecture~2.6.
The endpoints reduce to the familiar one-dimensional examples, whereas
the open intermediate interval requires two dimensions.

The result concerns a single step and the two performance ratios stated
in Theorem~\ref{thm:main}. It does not determine the sharp multistep
constant. In particular, one-step extremizers cannot simply be joined
without checking the compatibility of every interpolation constraint
across iterations. Establishing such compatibility is a separate
problem.

%% file: sections/disclosure.tex
\section{Disclosure}
\label{sec:disclosure}
The proof strategy is produced by OpenAI’s GPT-5.6 Sol Ultra through Codex in
response to prompts from the author. Codex was also used to revise
the exposition and prepare the LaTeX manuscript. The author selected the
problem, directed the interactions and revisions, and is the sole named author. The AI system is acknowledged as a reasoning and writing tool, not
as an author. This disclosure is not a substitute for independent expert
mathematical review.

%% file: sections/appendix.tex
\section{An exact example and supplementary verification}\label{sec:verification}
\subsection{A fully explicit interior example}
Take
\[
 h=1,\qquad \delta=\frac{\sqrt{11}}4,\qquad z=\frac14.
\]
These parameters are in the open intermediate regime. Indeed,
$\delta>1/2$ gives $\hL<1$, while \eqref{eq:threshold-ranges} gives
$\hR>1$. Here $a=0$, and the polynomial factors as
\[
 p(t)=t(t^2+t-5/16)=t(t-1/4)(t+5/4).
\]
Thus $z=1/4$ is its largest positive root. The construction gives
\[
 T=\frac1{\sqrt2},\quad v=0,\quad
 u=\frac1{16\sqrt2},\quad c=\frac18,\quad s=\frac1{32}.
\]
Consequently,
\[
 \begin{gathered}
 f(x,y)=\frac{x^2}{2}+H_{1/8}(y),\\
 x_0=(0,5/32),\qquad
 x_1=(1/(16\sqrt2),1/8),\qquad
 d=(-1/(16\sqrt2),1/32).
 \end{gathered}
\]
The gradients are
$g_0=(0,1/8)$ and $g_1=(1/(16\sqrt2),1/8)$, and $x_1=x_0-d$.
The error satisfies
\[
 \norm{d-g_0}^2
 =\frac1{512}+\frac9{1024}
 =\frac{11}{1024}
 =\frac{11}{16}\cdot\frac1{64}
 =\delta^2\norm{g_0}^2.
\]
The exact values are
\[
 f_\ast=0,\qquad f_0=\frac{12}{1024},\qquad
 f_1=\frac9{1024},\qquad \norm{g_1}^2=\frac{18}{1024}.
\]
For example, $f_0=c(5/32)-c^2/2=5/256-1/128=3/256$.
It follows that the gap ratio equals $3/2$ and the decrease ratio equals
$6$, exactly as predicted.

The scalar envelope in Proposition~\ref{prop:one-dimensional} is
\[
 W_1(1,\sqrt{11}/4)
 =\max\left\{\frac4{6-\sqrt{11}},\frac{11}{8}\right\}<\frac32.
\]
For the first term, the inequality is equivalent to
$3\sqrt{11}<10$, which holds because $99<100$; for the second it is
$11<12$. This gives a concrete strict dimension gap without numerical
approximation.

\subsection{Normalizing the objective gap}
The construction can have any prescribed positive initial gap $\Gamma$.
For the normalized example, let $f_0>0$ be as in
\eqref{eq:f0-factor}, and choose $\tau=\sqrt{\Gamma/f_0}$.
Define
\[
 \widehat f(x)=\tau^2 f(x/\tau),\qquad
 \widehat x_i=\tau x_i,\qquad \widehat d=\tau d.
\]
Then $\nabla\widehat f(x)=\tau\nabla f(x/\tau)$, so
\[
 \norm{\nabla\widehat f(x)-\nabla\widehat f(y)}
 \le\tau\norm{x/\tau-y/\tau}=\norm{x-y}.
\]
Convexity is preserved by a positive multiplication and a linear
change of variables. The update and relative error condition are
preserved, the minimum is still zero, and
$\widehat f(\widehat x_0)=\tau^2f_0=\Gamma$.
Both gradient norms squared and objective differences scale by
$\tau^2$, so neither ratio changes.
For smoothness constant $L$, first multiply the function and direction
by $L$, as in Section~\ref{sec:construction}, and choose the spatial
scale using its initial gap $Lf_0$.

\subsection{Verification Results}
The accompanying program \texttt{verify.py} uses exact symbolic algebra
to check twelve identities, including the source-cubic substitution,
the objective and gradient-norm factorizations, the error residual,
the transformation of both published constants, and the complete
nonnegative-square certificate. An identity check establishes equality
of algebraic expressions in the program; the sign and existence
arguments remain the analytic arguments in
Section~\ref{sec:parameters}.

The program additionally checks 74 parameter instances at 100 decimal
digits. They include both interval endpoints, points near the endpoints,
error levels as small as $10^{-8}$ and as large as $1-10^{-8}$, and
four instances with $h=1$. It checks the parameter inequalities, the
root selection, the relative error condition, both performance ratios,
and the vanishing certificate square.
All checks passed in the run accompanying this draft. The largest
normalized numerical residual was less than $2\cdot10^{-93}$.
The text file \texttt{verification\_results.txt} records that run.

These finite numerical checks are not used to infer a statement for
all real parameters. In particular they are not a substitute for
the root-selection proof, the endpoint calculations, or the
nonnegative-square argument.

% \subsection{Provenance and reproducibility}
% The source archive follows the supplied article template, retaining its
% document class, font setup, margins, theorem environments, section
% inputs, and bibliography workflow. The proofs are contained entirely
% in the LaTeX source. To compile, run \texttt{pdflatex main.tex},
% then \texttt{bibtex main}, then \texttt{pdflatex main.tex} twice.

% The conjecture's first appearance used here is
% \cite{VG2025v1}; its current full-text formulation is
% \cite{VG2025v2}. The accepted manuscript was also checked through
% the authors' institutional repository.\footnote{\href{https://research.dial.uclouvain.be/server/api/core/bitstreams/b0fa1ced-8e25-47e8-a18a-15ce5663cf08/content}{Accepted manuscript in the UCLouvain institutional repository.}}
% All source-status checks were performed on September 4, 2026.
% No claim is made that a literature search excludes unpublished work.

%% file: references.bib
@misc{VG2025v1,
 author={Pierre Vernimmen and Fran\c{c}ois Glineur},
 title={Worst-case convergence analysis of relatively inexact gradient descent on smooth convex functions},
 year={2025},
 howpublished={arXiv:2506.17145v1, June 20},
 note={Conjecture 2.6},
 url={https://arxiv.org/html/2506.17145v1}
}

@misc{VG2025v2,
 author={Pierre Vernimmen and Fran\c{c}ois Glineur},
 title={Worst-case convergence analysis of relatively inexact gradient descent on smooth convex functions},
 year={2025},
 howpublished={arXiv:2506.17145v2, September 11},
 note={Theorem 2.3 and Conjecture 2.6; accessed September 4, 2026},
 url={https://arxiv.org/html/2506.17145v2}
}

@article{VG2026,
 author={Pierre Vernimmen and Fran\c{c}ois Glineur},
 title={Worst-case convergence analysis of relatively inexact gradient descent on smooth convex functions},
 journal={Neurocomputing}, volume={660}, pages={131883}, year={2026},
 doi={10.1016/j.neucom.2025.131883},
 url={https://doi.org/10.1016/j.neucom.2025.131883}
}

@article{DT2014,
 author={Yoel Drori and Marc Teboulle},
 title={Performance of first-order methods for smooth convex minimization: a novel approach},
 journal={Mathematical Programming}, volume={145}, pages={451--482}, year={2014},
 doi={10.1007/s10107-013-0653-0},
 url={https://doi.org/10.1007/s10107-013-0653-0}
}

@article{THG2017,
 author={Adrien B. Taylor and Julien M. Hendrickx and Fran\c{c}ois Glineur},
 title={Smooth strongly convex interpolation and exact worst-case performance of first-order methods},
 journal={Mathematical Programming}, volume={161}, pages={307--345}, year={2017},
 doi={10.1007/s10107-016-1009-3},
 url={https://arxiv.org/abs/1502.05666}
}
